\documentclass[10pt,reqno]{amsart}

\usepackage{a4wide}
\usepackage[all]{xy}
\usepackage{upgreek}
\usepackage{dsfont} 
\usepackage{graphicx} 
\usepackage{amssymb}
\usepackage{graphics}
\usepackage[utf8]{inputenc}
\usepackage{t1enc}

\newtheorem{proposition}{Proposition}
\newtheorem{theorem}[proposition]{Theorem}
\newtheorem{lemma}[proposition]{Lemma}
\newtheorem{corollary}[proposition]{Corollary}

\newcommand{\cst}{\ensuremath{\mathrm{C}^*}}
\newcommand{\ii}{\mathrm{i}}
\newcommand{\comp}{\circ}
\newcommand{\I}{\mathds{1}}
\newcommand{\tens}{\otimes}
\newcommand{\hh}[1]{\widehat{#1}}
\renewcommand{\Bar}[1]{\overline{#1}}
\renewcommand{\tt}{\scriptscriptstyle{\mathsf{T}}}
\newcommand{\ph}{\varphi}
\newcommand{\id}{\mathrm{id}}

\newcommand{\GG}{\mathbb{G}}
\newcommand{\ZZ}{\mathbb{Z}}
\newcommand{\TT}{\mathbb{T}}
\newcommand{\KK}{\mathbb{K}}
\newcommand{\CC}{\mathbb{C}}
\newcommand{\RR}{\mathbb{R}}
\newcommand{\UU}{\mathbb{U}}
\newcommand{\HH}{\mathbb{H}}

\newcommand{\bh}{\boldsymbol{h}}
\newcommand{\bchi}{\boldsymbol{\upchi}}

\newcommand{\cT}{\mathcal{T}}
\newcommand{\cK}{\mathcal{K}}
\newcommand{\cQ}{\mathcal{Q}}
\newcommand{\cH}{\mathcal{H}}

\newcommand{\sN}{\mathsf{N}}
\newcommand{\sK}{\mathsf{K}}
\newcommand{\sH}{\mathsf{H}}

\DeclareMathOperator{\C}{C}
\DeclareMathOperator{\B}{B}
\DeclareMathOperator{\Pol}{Pol}
\DeclareMathOperator{\Irr}{Irr}
\DeclareMathOperator{\Ltwo}{\mathnormal{L}_2}
\DeclareMathOperator{\Linf}{\mathnormal{L}_\infty}
\DeclareMathOperator{\HS}{HS}
\DeclareMathOperator{\Tr}{Tr}
\DeclareMathOperator{\spec}{Sp}

\newcommand{\uu}{\text{\tiny{u}}}
\renewcommand{\ll}{\text{\tiny{L}}}
\renewcommand{\ss}{\mathnormal{s}}
\newcommand{\wW}{\text{\reflectbox{$\mathds{W}$}}\:\!} 

\begin{document}

\author{Jacek Krajczok}
\address{Institute of Mathematics of the Polish Academy of Sciences, Warsaw, Poland}
\email{jkrajczok@impan.pl}

\author{Piotr M.~So{\l}tan}
\address{Department of Mathematical Methods in Physics, Faculty of Physics, University of Warsaw, Poland}
\email{piotr.soltan@fuw.edu.pl}

\begin{abstract}
We show that the quantum disk, i.e.~the quantum space corresponding to the Toeplitz $\mathrm{C}^*$-algebra does not admit any compact quantum group structure. We prove that if such a structure existed the resulting compact quantum group would simultaneously be of Kac type and not of Kac type. The main tools used in the solution come from the theory of type $\mathrm{I}$ locally compact quantum groups, but also from the theory of operators on Hilbert spaces.
\end{abstract}

\title{The quantum disk is not a quantum group}




\maketitle

\section{Introduction}

It is well known that some important topological spaces do not admit a structure of a topological group. This is the case, for example, for all spheres of dimension different from $0$, $1$ and $3$. The proof of this result is surprisingly non-trivial as it involves various techniques of algebraic topology (cf.~\cite[Section B.4]{hatcher}, see also \cite{short} for a concise account of this result and a short proof that even-dimensional spheres are not topological groups).

An analogous question can be asked about \emph{quantum spaces}, which are ``virtual'' objects corresponding to non-commutative \cst-algebras under an extension of Gelfand's duality. This question has been addressed in \cite{without,when} for some well known quantum spaces. In particular it turns out that the quantum torus (\cite{connes,rieffel}) does not admit a structure of a compact quantum group, while its ``two fold covering'' does (\cite{qdt}). Also a number of well known quantum spheres are not compact quantum groups. Moreover the quantum space of all continuous maps from the quantum space underlying the \cst-algebra $M_2(\CC)$ into the classical two-point space does not admit a compact quantum group structure -- like the case of quantum tori, this too is in stark contrast to the classical situation (\cite[Section 5]{qmqs}).

In this paper we show that the quantum disk $\UU$ introduced in \cite{KLqdisk} (and denoted there by $U_\mu$) does not admit a structure of a compact quantum group. The quantum disk is a very well studied quantum space (see e.g.~\cite{HaMaSzym,SchmuWag,BrzDab,Qudd}). The \cst-algebra $\C(\UU)$ playing the role of the algebra of continuous functions on the quantum disk was shown in \cite[Theorem III.5]{KLqdisk} to be isomorphic to the \emph{Toeplitz algebra} $\cT$, i.e.~the \cst-algebra of operators on $\ell_2(\ZZ_+)$ generated by the operator $\boldsymbol{\mathsf{s}}\in\B(\ell_2(\ZZ_+))$ which shifts the standard basis $(e_n)_{n\in\ZZ_+}$: $\boldsymbol{\mathsf{s}}e_n=e_{n+1}$.

It is impossible to overestimate the importance of the Toeplitz algebra for non-commutative geometry (particularly $K$-theory, cf.~e.g.~\cite[Chapter 4]{CuMeRo}). The features of $\cT$ important for our purposes are that $\cT$ contains the ideal $\cK$ of compact operators on $\ell_2(\ZZ_+)$ and we have the short exact sequence
\[
\xymatrix{0\ar[r]&\cK\ar[r]&\cT\ar[r]^(.43)\ss&\C(\TT)\ar[r]&0},
\]
where $\C(\TT)$ is the algebra of all continuous function on the circle (\cite[Chapter V]{Davidson}). Moreover $\cK$ is an essential ideal in $\cT$ (cf.~\cite[Theorem 1]{Coburn} or \cite[Exercise 3.H]{WeO}) and since $\cK$ is a simple \cst-algebra, this means that $\cK$ is contained in any non-trivial ideal of $\cT$. We will also be using the fact that $\cT$ is a type $\mathrm{I}$ \cst-algebra (see e.g.~\cite[Section IV.1]{blackadar}) and its irreducible representations are indexed by $\{\bullet\}\cup\TT$, where the distinguished point corresponds to the identity representation on $\ell_2(\ZZ_+)$ and the remaining representations are one-dimensional -- given by compositions of the ``symbol map'' $\ss\colon\cT\to\C(\TT)$ and evaluations in the points of $\TT$. For $\lambda\in\TT$ we will denote the corresponding one-dimensional representation (character) of $\cT$ by $\ss_\lambda$.

We will be using the theory of compact and locally compact quantum groups. We refer the reader to the excellent book \cite{NTbook} and the original expositions in \cite{pseudogroups,cqg} for the details of the theory of compact quantum groups and to \cite{kv,VanDaele} for the theory of locally compact quantum groups. Our strategy is to assume that there exists a compact quantum group $\GG$ such that $\C(\GG)$ is isomorphic to $\cT$ and arrive at a contradiction. More precisely, the contradiction will come from the simple observation that $\GG$ cannot be of Kac type (Subsection \ref{notKac}) and the fact that it also is of Kac type (Section \ref{KacType}).

Our main technical tools will come from the theory of so called \emph{type $\mathrm{I}$ locally compact quantum groups}. We will recall the necessary and most important results of this theory in Section \ref{typeI}, but refer the reader to more detailed expositions \cite{DesmedtPhd,JKcoam,modular}.

For a possibly unbounded closed densely defined operator $T$ on a Hilbert space $\cH$ we will use the symbol $\spec(T)$ to denote the spectrum of $T$. Moreover we will use the notation from e.g.~\cite[Page 10]{so3} according to which $\cH(T=q)$ will denote the eigenspace of $T$ for the eigenvalue $q$ (or $\{0\}$ if $q$ is not an eigenvalue of $T$) and $\bchi(T=q)$ will stand for the orthogonal projection onto $\cH(T=q)$.

As we already explained we will work under the assumption that $\GG$ is a compact quantum group with $\C(\GG)$ isomorphic to $\cT$. The isomorphism $\C(\GG)\to\cT$ will be denoted by $\pi_{\bullet}$ while the remaining irreducible representations are $\{\ss_\lambda\comp\pi_\bullet\}_{\lambda\in\TT}$. We also let $\bh$ denote the Haar measure of $\GG$.

\section{First observations}\label{first}

In this section we discuss the properties of the compact quantum group $\GG$ which are easily obtainable from the basic knowledge about the structure of $\C(\GG)$. Some of these considerations have already appeared in \cite[Section 6]{when}. For this we let $\sK=\pi_\bullet^{-1}(\cK)$. Then $\sK$ is an ideal in $\C(\GG)$ which is contained in any non-trivial ideal and clearly $\sK$ is isomorphic to the algebra of compact operators on a separable Hilbert space.

\subsection{The Haar measure is faithful} The left kernel
\[
\sN=\bigl\{a\in\C(\GG)\,\bigr|\bigl.\,\bh(a^*a)=0\bigr\}
\]
of the Haar measure of $\GG$ is a closed two-sided ideal (\cite[Page 656]{pseudogroups}), so if it were not trivial, it would contain the ideal $\sK$. But then $\C(\GG)/\sN$ would be commutative which is impossible because both $\C(\GG)$ and $\C(\GG)/\sN$ are completions of the same $*$-algebra $\Pol(\GG)$ (\cite{coamen}), so one is commutative if and only if so is the other. It follows that $\sN=\{0\}$, i.e.~the Haar measure of $\GG$ is faithful.

\subsection{$\GG$ is co-amenable}

Co-amenability is a property of compact quantum groups described first in \cite{coamen}. One of its possible characterizations is that a compact quantum group $\KK$ is co-amenable if its Haar measure is faithful and $\C(\KK)$ admits a character. Since $\C(\GG)$ clearly admits many characters and we just showed above that $\bh$ is faithful, we see that $\GG$ is co-amenable. As a consequence (cf.~\cite{coamen}) $\C(\GG)$ is the universal enveloping \cst-algebra of $\Pol(\GG)$.

\subsection{$\GG$ is not of Kac type}\label{notKac}

A compact quantum group is of \emph{Kac type} if its Haar measure is a trace. This is clearly not the case with $\GG$ because the Haar measure $\bh$ is faithful, but since the ideal $\sK$ is spanned by commutators (\cite[Theorem 1]{PeaTopp}) any trace on $\C(\GG)$ must vanish on $\sK$ which would make it non-faithful.

One of the characterizations of compact quantum groups which are not of Kac type is that the family $\{f_z\}_{z\in\CC}$ of \emph{Woronowicz characters} of $\Pol(\GG)$ (see \cite[Definition 1.7.1]{NTbook}, cf.~\cite[Theorem 5.6]{pseudogroups}) is non-trivial. Moreover, since $z\mapsto{f_z}$ is holomorphic in an appropriate sense (\cite[Theorem 5.6]{pseudogroups}), the subfamily $\{f_{\ii{t}}\}_{t\in\RR}$ is also non-trivial.

It is known, that $\{f_{\ii{t}}\}_{t\in\RR}$ are $*$-characters of $\Pol(\GG)$, so since $\C(\GG)$ is the universal enveloping \cst-algebra, these functionals extend to a non-trivial continuous family of characters of $\C(\GG)$. Moreover this family of functionals forms a non-trivial compact group $F_{\text{\tiny{W}}}$ in the weak${}^*$-topology of $\C(\GG)^*$ which is a continuous image of $\RR$ under the map $t\mapsto{f_{\ii{t}}}$. Furthermore $F_{\text{\tiny{W}}}$ is a subgroup of the group $\widetilde{\GG}$ of all characters of $\C(\GG)$ (cf.~\cite[Theorem 3.12]{Gtilde}). But as a topological space the latter is simply $\TT$. So the group $F_{\text{\tiny{W}}}$ is a non-trivial connected compact subgroup of the group $\widetilde{\GG}$ which is topologically a circle. It follows that $F_{\text{\tiny{W}}}$ must be equal to all of $\widetilde{\GG}$. Since $\pi_\bullet$ is an isomorphism we have
\[
\widetilde{\GG}=\bigl\{\ss_\lambda\comp\pi_\bullet\,\bigr|\bigl.\,\lambda\in\TT\bigr\}
\]
and consequently there is a surjective mapping $\RR\ni{t}\mapsto\lambda(t)\in\TT$ such that $f_{\ii{t}}\comp\pi_\bullet^{-1}=\ss_{\lambda(t)}$.

\begin{proposition}\label{pierwsze}
An operator $A\in\cT$ is compact if and only if $f_{\ii{t}}(\pi_\bullet^{-1}(A))=0$ for all $t\in\RR$ .
\end{proposition}

\begin{proof}
Clearly $A\in\cK$ if and only if $\ss_\lambda(A)=0$ for all $\lambda$, so by the remarks above, $A$ is compact if and only if $A$ is mapped to zero by $f_{\ii{t}}\comp\pi_\bullet^{-1}$ for all $t\in\RR$.
\end{proof}

Following the standard conventions we let $\Irr(\GG)$ denote the set of equivalence classes of irreducible representations of $\GG$ and for each $\alpha\in\Irr(\GG)$ we fix a unitary representation $U^\alpha$ in the class $\alpha$ acting on a Hilbert space $\sH_\alpha$ of (finite) dimension $n_\alpha$. Furthermore we let $\uprho_\alpha$ denote the positive operator on $\sH_\alpha$ encoding the modular properties of $\bh$ (\cite[Theorem 5.4]{pseudogroups}, \cite[Theorem 1.4.4]{NTbook}). Furthermore for each $\alpha$ we fix an orthonormal basis of $\sH_\alpha$ in which the matrix of $\uprho_\alpha$ is diagonal. Finally we let $U^\alpha_{i,j}$ be the matrix elements of $U$ with respect to this basis.

\begin{proposition}
For any $\alpha\in\Irr(\GG)$ and $i,j\in\{1,\dotsc,n_\alpha\}$ the operator $\pi_\bullet(U^\alpha_{i,j})$ is compact if and only if $i\neq{j}$. Moreover $\pi_\bullet(U^\alpha_{i,i})$ is a Fredholm operator.
\end{proposition}

\begin{proof}
Let $\uprho_{\alpha,1},\dotsc,\uprho_{\alpha,n_\alpha}$ be the eigenvalues of $\uprho_\alpha$, so that
\[
\uprho_\alpha=\operatorname{diag}(\uprho_{\alpha,1},\dotsc,\uprho_{\alpha,n_\alpha})
\]
in the fixed basis of $\sH_\alpha$. Then by the definition of the Woronowicz characters (\cite[Theorem 5.6]{pseudogroups}, \cite[Definition 1.7.1]{NTbook}) we have
\[
f_{\ii{t}}(U^\alpha_{i,j})=\delta_{i,j}(\uprho_{\alpha,i})^{\ii{t}},\qquad{t}\in\RR
\]
and the first part of the proposition follows immediately from Proposition \ref{pierwsze}. For the second part note that from the unitarity of $U^\alpha$:
\[
\sum_{j=1}^{n_\alpha}U^\alpha_{i,j}{U^\alpha_{i,j}}^*=\I
\]
and from the fact that $\pi_\bullet(U^\alpha_{i,j})$ are compact for $i\neq{j}$ it follows that $\ss\bigl(\pi_\bullet(U^\alpha_{i,i})\bigr)$ is unitary, so the operator $\pi_\bullet(U^\alpha_{i,i})$ is Fredholm.
\end{proof}

\section{Type I locally compact quantum groups}\label{typeI}

In this section we recall the most important elements of the theory of locally compact quantum groups of type $\mathrm{I}$. A locally compact quantum group $\HH$ is of \emph{type $\mathrm{I}$} if $\C_0^\uu(\hh{\HH})$ is a type $\mathrm{I}$ \cst-algebra (cf.~\cite[Section IV.1]{blackadar}). This \cst-algebra is responsible for the unitary representations of $\HH$ (cf.~\cite{univ,mmu}).

The study of type $\mathrm{I}$ locally compact quantum group was initiated in the thesis \cite{DesmedtPhd} and continued e.g.~in \cite{CaspersPhD,JKcoam}. We will be using the results of \cite{DesmedtPhd,JKcoam} as well as some of their extensions from \cite{modular}.

The most important feature of these quantum groups is the existence of the \emph{Plancherel measure}, i.e.~a standard measure $\mu$ on the spectrum of $\C_0^\uu(\hh{\HH})$, denoted later by $\Irr(\HH)$,
a measurable field $\{\sH_x\}_{x\in\Irr(\HH)}$ of Hilbert spaces,
a measurable field of representations $\{\pi_x\}_{x\in\Irr(\HH)}$ of $\C_0^\uu(\hh{\HH})$,
a measurable field of strictly positive self-adjoint operators $\{D_x\}_{x\in\Irr(\HH)}$ and a unitary operator
\[
\cQ_{\ll}\colon\Ltwo(\hh{\HH})\longrightarrow\int_{\Irr(\HH)}^\oplus\HS(\sH_x)\,d\mu(x)
\]
such that
\[
\cQ_{\ll}\Linf(\hh{\HH})\cQ_{\ll}^*=\int_{\Irr(\HH)}^\oplus\bigl(\B(\sH_x)\tens\I_{\Bar{\sH_x}}\bigr)\,d\mu(x)
\]
(here for any Hilbert space $\cH$ we identify $\HS(\cH)$ with $\cH\tens\Bar{\cH}$) and for positive $a\in\C_0^\uu(\hh{\HH})$ the value of the left Haar measure $\hh{\ph}$ of $\hh{\HH}$ on $a$  is
\[
\hh{\ph}(a)=\int_{\Irr(\HH)}\Tr\bigl(\pi_x(a)\cdot{D_x^{-2}}\bigr)\,d\mu(x).
\]
We refer the reader to \cite[Section 3.4]{DesmedtPhd} and
\cite{JKcoam}
for details.

The next lemma establishes the form of the image of an element of the universal \cst-algebra $\C_0^\uu(\hh{\HH})$ under the \emph{reducing morphism} (cf.~\cite[Section 2]{univ}, \cite[Definition 35]{mmu}) which was denoted  by $\hat{\pi}$ in \cite[Notation 2.15]{univ} and by $\Lambda_{\hh{\HH}}$ in \cite{mmu}. In order to avoid conflict with the notation for GNS maps we will denote it by $\Pi_{\hh{\HH}}$.

\begin{lemma}\label{lem5}
Let $\HH$ be a type $\mathrm{I}$ locally compact quantum group such that $\C_0^\uu(\hh{\HH})$ is separable.  Let $\mu$ be the Plancherel measure for $\HH$ an let $\Pi_{\hh{\HH}}\colon\C_0^\uu(\hh{\HH})\to\C_0(\hh{\HH})$ be the reducing morphism for $\hh{\HH}$. Then for any $a\in\C_0^\uu(\hh{\HH})$ we have
\begin{equation}\label{PiH}
\cQ_{\ll}\Pi_{\hh{\HH}}(a)\cQ_{\ll}^*
=\int_{\Irr(\HH)}^\oplus\bigl(\pi_x(a)\tens\I_{\Bar{\sH_x}}\bigr)\,d\mu(x).
\end{equation}
\end{lemma}

\begin{proof}
Given $a\in\C_0^\uu(\hh{\HH})$ the element $\cQ_{\ll}\Pi_{\hh{\HH}}(a)\cQ_{\ll}^*$ can be written as
\[
\cQ_{\ll}\Pi_{\hh{\HH}}(a)\cQ_{\ll}^*=\int_{\Irr(\HH)}^{\oplus}\bigl(a_x\tens\I_{\Bar{\sH_x}}\bigr)\,d\mu(x).
\]
for some measurable field $\{a_\pi\}_{\pi\in\Irr(\HH)}$. By Desmedt's results (\cite{DesmedtPhd}) for any $\omega\in\Linf(\HH)_*$ we have
\[
\cQ_{\ll}\bigl((\omega\tens\id)(\mathrm{W})\bigr)\cQ_{\ll}^*=
\int_{\Irr(\HH)}^{\oplus}\bigl((\omega\tens\id)(U^{\pi_x})\tens\I_{\Bar{\sH_x}}\bigr)\,d\mu(x),
\]
where $\mathrm{W}\in\Linf(\HH)\bar{\tens}\Linf(\hh{\HH})$ is the Kac-Takesaki operator (\cite{kv,mmu,VanDaele}) and $U^{\pi_x}$ is the unitary representation of $\HH$ corresponding to $\pi_x$. Thus denoting by $\lambda^\uu$ the mapping
\[
\Linf(\HH)_*\ni\omega\longmapsto(\omega\tens\id)\wW\in\C_0^\uu(\hh{\HH})
\]
(where $\wW$ is the universal representation of $\HH$, cf.~\cite{mmu}) we have
\[
\cQ_{\ll}\Pi_{\hh{\HH}}\bigl(\lambda^\uu(\omega)\bigr)\cQ_{\ll}^*
=\int_{\Irr(\HH)}^\oplus\bigl(\pi_x(\lambda^\uu(\omega))\tens\I_{\Bar{\sH_x}}\bigr)\,d\mu(x).
\]
Both sides of the above equation are continuous with respect to $\lambda^\uu(\omega)$ (for the right hand side we can use \cite[Section 2.3, Proposition 4]{DixVN} because the range of $\lambda^\uu$ is dense in $\C_0^\uu(\hh{\HH})$) and \eqref{PiH} follows.
\end{proof}

In what follows we will use the above results for the discrete quantum group $\HH=\hh{\GG}$. Note that as $\GG$ is co-amenable, we have $\C^\uu(\GG)=\C(\GG)$.

\section{The Plancherel measure for the dual of $\GG$}

Let $\mu$ be the Plancherel measure for $\hh{\GG}$. This is a measure on $\Irr(\hh{\GG})=\{\bullet\}\cup\TT$ such that there is a unitary operator $\cQ_\ll\colon\Ltwo(\GG)\to\int^\oplus_{\Irr(\hh{\GG})}\HS(\sH_\pi)\,d\mu(\pi)$ and
\[
\cQ_{\ll}\Linf(\GG)\cQ_{\ll}^*=\int^\oplus_{\Irr(\hh{\GG})}\bigl(\B(\sH_x)\tens\I_{\Bar{\sH_x}}\bigr)\,d\mu(x).
\]
The subset $\TT$ of $\Irr(\hh{\GG})$ is measurable (because all representations belonging to $\TT$ are of the same dimension) and thus $\{\bullet\}$ is measurable. Moreover, since $\Linf(\GG)$ is non-commutative, we must have $\mu(\{\bullet\})>0$ and we can re-scale it, so that $\mu(\{\bullet\})=1$. It follows that
\begin{equation}\label{QlL2}
\cQ_{\ll}\bigl(\Ltwo(\GG)\bigr)=\HS(\sH_\bullet)\oplus\int^\oplus_{\TT}\HS(\sH_\lambda)\,d\mu(\lambda)
\end{equation}
and
\begin{equation}\label{MMM}
\cQ_{\ll}\Linf(\GG)\cQ_{\ll}^*=\bigl(\B(\sH_\bullet)\tens\I_{\Bar{\sH_\bullet}}\bigr)\oplus\int^\oplus_{\TT}\bigl(\B(\sH_\lambda)\tens\I_{\Bar{\sH_\lambda}}\bigr)\,d\mu(\lambda).
\end{equation}
We will denote $\cQ_{\ll}\Linf(\GG)\cQ_{\ll}^*$ by $M$ and the two the two summands in the decomposition \eqref{MMM} by $M_1$ and $M_2$ respectively, so that $M=M_1\oplus{M_2}$. We also let $\I_1$ and $\I_2$ be the units of $M_1$ and $M_2$. Note that $M_1$ is isomorphic to $\B(\sH_\bullet)$ and $M_2$ is commutative and in fact isomorphic to $\Linf(\TT,\left.\mu\right|_\TT)$.

\begin{lemma}\label{actonM}
Let $\beta$ be an automorphism of $M$. Then $\beta$ preserves the decomposition $M=M_1\oplus{M_2}$. In particular $\beta(\I_1)=\I_1$ and $\beta(\I_2)=\I_2$.
\end{lemma}

\begin{proof}
Let $E_1$ and $E_2$ be the projections of $M$ onto the two summands $M_1$ and $M_2$ (so that $E_i(x)=\I_ix$ for any $x\in{M}$). The map $M_1\ni{y}\mapsto{E_2}\bigl(\beta(y)\bigr)\in{M_2}$ is a normal $*$-homomorphism $M_1\to{M_2}$ which must be zero because $M_1$ is a factor and $M_2$ is commutative. It follows that $\beta(M_1)\subset{M_1}$ and since this is also true for the automorphism $\beta^{-1}$, we have $\beta^{-1}(M_1)\subset{M_1}$ and acting with $\beta$ on both sides gives $M_1\subset\beta(M_1)$. It follows that $\beta$ restricts to an automorphisms of $M_1$, so it must preserve $\I_1$. Clearly if there were $z\in{M_2}$ such that $E_1\bigl(\beta(z)\bigr)\neq{0}$ then $z=\beta^{-1}(\beta(z))=\beta^{-1}\bigl(E_1(\beta(z))+E_2(\beta(z))\bigr)=\beta^{-1}\bigl(E_1(\beta(z))\bigr)+\beta^{-1}\bigl(E_2(\beta(z))\bigr)$ which is a contradiction because $\beta^{-1}$ is injective and preserves $M_1$, so $\beta^{-1}\bigl(E_1(\beta(z))\bigr)\neq{0}$ and hence $E_1(z)\neq{0}$. It follows that $\beta$ preserves $M_2$ and consequently $\beta(\I_2)=\I_2$.
\end{proof}

One way to use Lemma \ref{actonM} is to apply it to the scaling automorphisms (\cite[Page 32]{NTbook}) of $\Linf(\GG)$ transferred to $M$ via the unitary $\cQ_\ll$:
\[
\beta_t(a)=\cQ_{\ll}\tau_t\bigl(\cQ_{\ll}^*a\cQ_{\ll}\bigr)\cQ_{\ll}^*,\qquad{a}\in{M},\;t\in\RR.
\]
It follows that the one-parameter group $(\beta_t)_{t\in\RR}$ restricts to a one-parameter group of automorphisms of $M_1$. Now the scaling group $(\tau_t)_{t\in\RR}$ is implemented on $\Ltwo(\GG)$ by a one-parameter group $(P^{\ii{t}})_{t\in\RR}$ of unitary operators (\cite[Definition 6.9]{kv}, cf.~also \cite[Remarks after Proposition 5.13]{VanDaele}).

Thus we obtain a one-parameter group of automorphisms $(\alpha_t)_{t\in\RR}$ of $\B(\sH_\bullet)$ defined by
\[
\alpha_t(x)\tens\I_{\Bar{\sH_\bullet}}=\cQ_{\ll}\tau_t(\cQ_{\ll}^*(x\tens\I_{\Bar{\sH_\bullet}})\cQ_{\ll})\cQ_{\ll}^*,\qquad{x}\in\B(\sH_\bullet),\;t\in\RR
\]
and thus by \cite[Theorem 4.13]{kadison} there exists a strongly continuous one-parameter group of unitary operators $(A_t)_{t\in\RR}$ on $\sH_\bullet$ such that
\[
\alpha_t(x)\tens\I_{\Bar{\sH_\bullet}}=A_txA_{-t}\tens\I_{\Bar{\sH_\bullet}},\qquad{x}\in\B(\sH_\bullet),\;t\in\RR.
\]
Moreover the one-parameter group $(P^{\ii{t}})_{t\in\RR}$ also induces automorphisms of $\Linf(\GG)'\subset\B(\Ltwo(\GG))$ which can be transferred to automorphisms of $M'$ by the unitary $\cQ_\ll$. Clearly $M'={M_1}'\oplus{M_2}'$ and by a process analogous to the one presented for $M$ we obtain a $\sigma$-weakly continuous one-parameter group of automorphisms of ${M_1}'=\I_{\sH_\bullet}\tens\B(\Bar{\sH_\bullet})$ which yields a group $(\alpha'_t)_{t\in\RR}$ of automorphisms of $\B(\Bar{\sH_\bullet})$:
\[
\I_{\sH_\bullet}\tens\alpha'_t(y)=\cQ_{\ll}P^{\ii{t}}\bigl(\cQ_{\ll}^*(\I_{\sH_\bullet}\tens{y})\cQ_{\ll}\bigr)P^{-\ii{t}}\cQ_{\ll}^*,\qquad{y}\in\B(\Bar{\sH_\bullet}),\;t\in\RR.
\]
It follows that there is a strongly continuous one-parameter group of unitary operators $(B_t)_{t\in\RR}$ on $\sH_\bullet$ such that
\[
\I_{\sH_\bullet}\tens\alpha'_t(y)=\I_{\sH_\bullet}\tens{B_t^{\tt}yB_{-t}^{\tt}},\qquad{y}\in\B(\Bar{\sH_\bullet}),\;t\in\RR
\]
(for future notational convenience we choose to consider the group $(B_t)_{t\in\RR}$ on $\sH_\bullet$ and work with the transposed operators on $\Bar{\sH_\bullet}$).

Clearly the group $(\beta_t)_{t\in\RR}$ is implemented by the unitary operators $(\cQ_{\ll}P^{\ii{t}}\cQ_{\ll}^*)_{t\in\RR}$ and since the group preserves the projections $\I_1$ and $\I_2$, these operators are block-diagonal in the decomposition \eqref{QlL2}. It follows that for any $x\in\B(\sH_\bullet)$ and $y\in\B(\Bar{\sH_\bullet})$ we have
\[
(\I_1\cQ_{\ll}P^{\ii{t}}\cQ_{\ll}^*)(x\tens{y})(\I_1\cQ_{\ll}P^{-\ii{t}}\cQ_{\ll}^*)
=\alpha_t(x)\tens\alpha'_t(y)=(A_t\tens{B_t^{\tt}})(x\tens{y})(A_{-t}\tens{B_{-t}^{\tt}})
\]
for all $t\in\RR$. This means that
\[
\I_1\cQ_{\ll}P^{\ii{t}}\cQ_{\ll}^*=\lambda_t(A_t\tens{B_t^{\tt}}),\qquad{t}\in\RR
\]
for some complex numbers $(\lambda_t)_{t\in\RR}$ of absolute value $1$ depending continuously on $t$. Moreover, since the two one-parameter groups $(\cQ_{\ll}P^{\ii{t}}\cQ_{\ll}^*)_{t\in\RR}$ and $(A_t\tens{B_t^{\tt}})_{t\in\RR}$ obviously commute, $t\mapsto\lambda_t$ is also a homomorphism, so defining $\widetilde{A}_t=\lambda_tA_t$ we obtain a strongly continuous one-parameter group of unitaries such that
\[
\I_1\cQ_{\ll}P^{\ii{t}}\cQ_{\ll}^*=\widetilde{A}_t\tens{B_t^{\tt}},\qquad{t}\in\RR.
\]

In the proof of the next proposition we will use a technical lemma.

\begin{lemma}
Let $\sH$ be a Hilbert space and let
\[
J\colon\sH\tens\Bar{\sH}\ni\xi\tens\Bar{\eta}\longmapsto\eta\tens\Bar{\xi}\in\sH\tens\Bar{\sH}.
\]
Let $(a_t)_{t\in\RR}$ and $(b_t)_{t\in\RR}$ be strongly continuous one-parameter groups of unitary operators on $\sH$ and assume that $J(a_t\tens{b_t^{\tt}})=(a_t\tens{b_t^{\tt}})J$ for all $t\in\RR$. Then for all $t$ we have $a_t=b_{-t}$.
\end{lemma}

\begin{proof}
On one hand we have
\begin{equation}\label{JabJ0}
J(a_t\tens{b_t^{\tt}})J=a_t\tens{b_t^{\tt}},
\end{equation}
so for any $r,s\in\RR$ and any $x\in\B(\sH)$
\begin{equation}\label{duzoab}
\begin{aligned}
(a_s\tens{b_s^{\tt}})J(a_t\tens{b_t^{\tt}})J(x\tens\I_{\Bar{\sH}})&J(a_{-t}\tens{b_{-t}^{\tt}})J(a_{-s}\tens{b_{-s}^{\tt}})\\
&=(a_s\tens{b_s^{\tt}})(a_t\tens{b_t^{\tt}})(x\tens\I_{\Bar{\sH}})(a_{-t}\tens{b_{-t}^{\tt}})(a_{-s}\tens{b_{-s}^{\tt}})\\
&=(a_t\tens{b_t^{\tt}})(a_s\tens{b_s^{\tt}})(x\tens\I_{\Bar{\sH}})(a_{-s}\tens{b_{-s}^{\tt}})(a_{-t}\tens{b_{-t}^{\tt}})\\
&=J(a_t\tens{b_t^{\tt}})J(a_s\tens{b_s^{\tt}})(x\tens\I_{\Bar{\sH}})(a_{-s}\tens{b_{-s}^{\tt}})J(a_{-t}\tens{b_{-t}^{\tt}})J.
\end{aligned}
\end{equation}
On the other hand
\begin{equation}\label{JabJ}
J(a_t\tens{b_t^{\tt}})J=b_{-t}\tens{a_{-t}^{\tt}},\qquad{t}\in\RR,
\end{equation}
so \eqref{duzoab} reads
\[
(a_s\tens{b_s^{\tt}})(b_{-t}\tens{a_{-t}^{\tt}})(x\tens\I_{\Bar{\sH}})(b_{t}\tens{a_{t}^{\tt}})(a_{-s}\tens{b_{-s}^{\tt}})
=(b_{-t}\tens{a_{-t}^{\tt}})(a_s\tens{b_s^{\tt}})(x\tens\I_{\Bar{\sH}})(a_{-s}\tens{b_{-s}^{\tt}})(b_{t}\tens{a_{t}^{\tt}}).
\]
Thus for any $x$ and all $s,t$ we have $a_sb_{-t}xb_ta_{-s}=b_{-t}a_sxa_{-s}b_t$, i.e.~$a_{-s}b_ta_sb_{-t}$ commutes with all $x\in\B(\sH)$.

Therefore there exists a continuous family $\{\lambda_{t,s}\}_{t,s\in\RR}$ of complex numbers of absolute value $1$ such that
\begin{equation}\label{attb}
a_{-s}b_ta_sb_{-t}=\lambda_{t,s}\I_{\sH},\qquad{t,s}\in\RR.
\end{equation}

Note now that in view of the canonical isomorphism $\B(\sH\tens\Bar{\sH})\cong\B(\HS(\sH))$ given by
\[
\B(\sH\tens\Bar{\sH})\ni{x\tens{y^{\tt}}}\longmapsto
\bigl(S\mapsto{xSy}\bigr)\in\B\bigl(\HS(\sH)\bigr)
\]
equations \eqref{JabJ0} and \eqref{JabJ} mean that
\[
a_tSb_t=b_{-t}Sa_{-t},\qquad{t}\in\RR,\;S\in\HS(\sH)
\]
which by strong density of $\HS(\sH)$ in $\B(\sH)$ gives
\[
a_tb_t=b_{-t}a_{-t},\qquad{t}\in\RR
\]
In particular for each $t$ the operator $a_tb_t$ is self-adjoint, and taking adjoints of this for $-t$ instead of $t$ we see that also $b_ta_t$ is self-adjoint for all $t$.

Therefore inserting $s=-t$ in \eqref{attb} gives $a_tb_t=\lambda_{t,-t}b_ta_t$ and since $a_tb_t$ and $b_ta_t$ are self-adjoint, $t\mapsto\lambda_{t,-t}$ is continuous and $\lambda_{0,0}=1$, we obtain $\lambda_{t,-t}=1$ for all $t$.

Consequently $a_tb_t=b_ta_t=(b_ta_t)^*=a_{-t}b_{-t}$, so that $b_{2t}=a_{-2t}$ for all $t$.
\end{proof}

\begin{proposition}
With the notation introduced above we have $\widetilde{A}_t=B_{-t}$ for all $t\in\RR$ .
\end{proposition}

\begin{proof}
We will use the fact that for all $t$ we have $P^{\ii{t}}J_\GG=J_{\GG}P^{\ii{t}}$, where $J_\GG$ is the modular conjugation of the Haar measure of $\GG$ (\cite[Remarks after Theorem 5.17]{VanDaele}).
Moreover, by \cite[Theorem 3.4.5 (3)]{DesmedtPhd}, the operator
\[
\cQ_{\ll}J_\GG\cQ_{\ll}^*\colon\HS(\sH_\bullet)\oplus\int^\oplus_{\TT}\HS(\sH_\lambda)\,d\mu(\lambda)
\longrightarrow\HS(\sH_\bullet)\oplus\int^\oplus_{\TT}\HS(\sH_\lambda)\,d\mu(\lambda)
\]
acts as
\[
(\xi\tens\Bar{\eta})\oplus\int^\oplus_\TT(\xi_\lambda\tens\Bar{\eta_\lambda})\,d\mu(\lambda)\longmapsto
(\eta\tens\Bar{\xi})\oplus\int^\oplus_\TT(\eta_\lambda\tens\Bar{\xi_\lambda})\,d\mu(\lambda).
\]
In particular $\cQ_{\ll}J_\GG\cQ_{\ll}^*$ is block-diagonal and
\[
\I_1\cQ_{\ll}J_\GG\cQ_{\ll}^*\colon\HS(\sH_\bullet)\ni(\xi\tens\Bar{\eta})\longmapsto
(\eta\tens\Bar{\xi})\in\HS(\sH_\bullet).
\]
Therefore
\begin{align*}
\I_1\cQ_{\ll}J_\GG\cQ_{\ll}^*\I_1(\widetilde{A}_t\tens{B_t^{\tt}})
&=\I_1\cQ_{\ll}J_{\GG}P^{\ii{t}}\cQ_{\ll}^*\I_1\\
&=\I_1\cQ_{\ll}P^{\ii{t}}J_{\GG}\cQ_{\ll}^*\I_1\\
&=(\widetilde{A}_t\tens{B_t^{\tt}})\I_1\cQ_{\ll}J_\GG\cQ_{\ll}^*\I_1
\end{align*}
for all $t\in\RR$.
\end{proof}

\begin{corollary}\label{zeStw1}
For each $t\in\RR$ the restriction of $\cQ_{\ll}P^{\ii{t}}\cQ_{\ll}^*$ to $\HS(\sH_\bullet)$ is equal to $B_{-t}\tens{B_t^{\tt}}$.
\end{corollary}

In what follows we let $B$ be the exponential of the infinitesimal generator of the one-parameter group $(B_t)_{t\in\RR}$, so that $B_t=B^{\ii{t}}$ for all $t$.

\section{Operators $D_\bullet$ and $B$}

Since $\{\bullet\}$ is a subset of $\Irr(\hh{\GG})$ of positive measure, the operator $D_\bullet$ is well defined (see Section \ref{typeI}). By the properties of Plancherel measure (Section \ref{typeI}) we have
\[
1=\bh(\I)=\Tr(D_\bullet^{-2})+\int_\TT\Tr(D_\lambda^{-2})\,d\mu(\lambda),
\]
so $D_\bullet^{-1}$ is a Hilbert-Schmidt operator, so in particular it is compact. It follows that $D_\bullet^{-1}\in\cT$. The eigenvalues of $D_\bullet^{-1}$ are of finite multiplicity, they form a countable subset of $\left]0,+\infty\right[$ and we have the norm convergent series
\[
D_\bullet^{-1}=\sum_{q\in\spec(D_\bullet^{-1})}q\bchi(D_\bullet^{-1}=q)
\]
(\cite[Section 5.2]{primer}).

\begin{proposition}
The operators $B$ and $D_\bullet$ strongly commute.
\end{proposition}

\begin{proof}
From the properties of the Plancherel measure (\cite[Theorem 3.4.5(4)]{DesmedtPhd}) we get
\[
\cQ_\ll\Lambda_{\bh}(\I)=D_\bullet^{-1}\oplus\int^\oplus_{\TT}D_\lambda^{-1}\,d\mu(\lambda)\in\HS(\sH_\bullet)\oplus\int^\oplus_{\TT}\HS(\sH_\lambda)\,d\mu(\lambda).
\]
Now we fix $t\in\RR$ and note that since $\tau_t(\I)=\I$, we have
\begin{align*}
D_\bullet^{-1}\oplus\int^\oplus_{\TT}D_\lambda^{-1}\,d\mu(\lambda)
&=\cQ_\ll\Lambda_{\bh}(\I)\\
&=\cQ_\ll\Lambda_{\bh}\bigl(\tau_t(\I)\bigr)\\
&=\cQ_{\ll}P^{\ii{t}}\Lambda_{\bh}(\I)\\
&=B_{-t}D_\bullet^{-1}B_t\oplus\bigl(\I_2(\cQ_{\ll}P^{\ii{t}}\cQ_{\ll}^*)\bigr)
\biggl(\int^\oplus_{\TT}D_\lambda^{-1}\,d\mu(\lambda)\biggr),
\end{align*}
which implies that $B_tD_\bullet^{-1}=D_\bullet^{-1}B_t$.
\end{proof}

\begin{corollary}\label{DBD}
The operator $B$ preserves the decomposition of $\sH_\bullet$ into eigenspaces of $D_\bullet^{-1}$:
\[
\sH_\bullet=\bigoplus_{q\in\spec(D_\bullet^{-1})}\sH_\bullet(D_\bullet^{-1}=q)
\]
so that
\[
B=\bigoplus_{q\in\spec(D_\bullet^{-1})}\bchi(D_\bullet^{-1}=q)B\bchi(D_\bullet^{-1}=q).
\]
\end{corollary}

\section{$\GG$ is of Kac type}\label{KacType}

We begin with three lemmas relating the structure of the compact quantum group $\GG$ to the decomposition of $\sH_\bullet$ into eigenspaces of $D_\bullet^{-1}$. The modular group of $\bh$ (cf.~\cite[Section 8]{cqg}) will be denoted by $(\sigma_t^{\bh})_{t\in\RR}$.

\begin{lemma}\label{prev}
For $a\in\cT$ and $t\in\RR$ we have
\begin{align*}
\pi_\bullet\bigl(\sigma_t^{\bh}(\pi_\bullet^{-1}(a))\bigr)&=D_\bullet^{-2\ii{t}}aD_\bullet^{2\ii{t}},\\
\pi_\bullet\bigl(\tau_t(\pi_\bullet^{-1}(a))\bigr)&=B_{-t}aB_t.
\end{align*}
\end{lemma}

\begin{proof}
It will be shown in \cite{modular} that the modular operator for $\bh$ transported via the unitary $\cQ_{\ll}$ acts as follows:
\[
\cQ_\ll\nabla_{\bh}^{\ii{t}}\cQ_{\ll}^*
=\bigl(D_\bullet^{-2\ii{t}}\tens(D_\bullet^{2\ii{t}})^{\tt}\bigr)\oplus
\int_{\TT}^{\oplus}\bigl(D_\lambda^{-2\ii{t}}\tens(D_\lambda^{2\ii{t}})^{\tt}\bigr)\,d\mu(\lambda).
\]
Take $a\in\C(\GG)$. Denoting $\ph_\lambda=\ss_\lambda\comp\pi_\bullet$ from Lemma \ref{lem5} we conclude that
\[
\cQ_{\ll}a\cQ_{\ll}^*=\bigl(\pi_\bullet(a)\tens\I_{\Bar{\sH_\bullet}}\bigr)\oplus
\int_{\TT}^{\oplus}\bigl(\ph_\lambda(a)\tens\I_{\Bar{\sH_\lambda}}\bigr)\,d\mu(\lambda),
\]
so
\begin{align*}
\bigl(\pi_\bullet\bigl(\sigma_t^{\bh}(a)\bigr)\tens\I_{\Bar{\sH_\bullet}}\bigr)\oplus&\int_{\TT}^{\oplus}\bigl(\ph_\lambda\bigl(\sigma_t^{\bh}(a)\bigr)\tens\I_{\Bar{\sH_\lambda}}\bigr)\,d\mu(\lambda)\\
&=\cQ_{\ll}\sigma_t^{\bh}(a)\cQ_{\ll}^*\\
&=\cQ_{\ll}\nabla_{\bh}^{\ii{t}}a\nabla_{\bh}^{-\ii{t}}\cQ_{\ll}^*\\
&=\bigl(D_\bullet^{-2\ii{t}}\pi_\bullet(a)D_\bullet^{2\ii{t}}\tens\I_{\Bar{\sH_\bullet}}\bigr)\oplus
\int_{\TT}^{\oplus}\bigl(\ph_\lambda(a)\tens\I_{\Bar{\sH_\lambda}}\bigr)\,d\mu(\lambda)
\end{align*}
(representations $\{\ph_\lambda\}$ are one-dimensional) and consequently
\[
\pi_\bullet\bigl(\sigma_t^{\bh}(\pi_\bullet^{-1}(a))\bigr)
=D_\bullet^{-2\ii{t}}aD_\bullet^{2\ii{t}},\qquad{a}\in\cT,\;t\in\RR.
\]

The second part of the lemma is proved analogously using Corollary \ref{zeStw1}.
\end{proof}

\begin{lemma}\label{lemShift}
For any $\alpha\in\Irr(\GG)$, $i\in\{1,\dotsc,n_\alpha\}$ and $q\in\spec(D_\bullet^{-1})$ the operator $\pi_\bullet(U^\alpha_{i,i})$ shifts the eigenspaces of $D_\bullet^{-1}$ as follows:
\[
\pi_\bullet(U^\alpha_{i,i})\sH_\bullet(D_\bullet^{-1}=q)\subset
\sH_\bullet\bigl(D_\bullet^{-1}=q\uprho_{\alpha,i}\bigr).
\]
\end{lemma}

\begin{proof}
For $t\in\RR$ and $\xi\in\sH_\bullet(D_\bullet^{-1}=q)$, by Lemma \ref{prev} we have
\[
D_\bullet^{\ii{t}}\pi_\bullet(U^\alpha_{i,i})\xi
=D_\bullet^{\ii{t}}\pi_\bullet(U^\alpha_{i,i})D_\bullet^{-\ii{t}}D_\bullet^{\ii{t}}\xi
=\pi_\bullet\bigl(\sigma_{-t/2}^{\bh}(U^\alpha_{i,i})\bigr)q^{-\ii{t}}\xi=q^{-\ii{t}}\uprho_{\alpha,i}^{-\ii{t}}\pi_\bullet(U^\alpha_{i,i})\xi
\]
(cf.~e.g.~\cite[Section 1.7]{NTbook}) which means that $\pi_\bullet(U^\alpha_{i,i})\xi\in\sH_\bullet\bigl(D_\bullet^{-1}=q\uprho_{\alpha,i}\bigr)$.
\end{proof}

Clearly for any $q\in\spec(D_\bullet^{-1})$ the operator $\bchi(D_\bullet^{-1}=q)B\bchi(D_\bullet^{-1}=q)$ is bounded and positive. We let $\Delta_q$ denote its spectrum:
\[
\Delta_q=\spec\bigl(\bchi(D_\bullet^{-1}=q)B\bchi(D_\bullet^{-1}=q)\bigr).
\]

\begin{lemma}
For any $\alpha\in\Irr(\GG)$, $i\in\{1,\dotsc,n_\alpha\}$, $q\in\spec(D_\bullet^{-1})$ and $c\in\Delta_q$ we have
\begin{equation}\label{zaw}
\pi_\bullet(U^\alpha_{i,i})
\sH_\bullet\bigl(\bchi(D_\bullet^{-1}=q)B\bchi(D_\bullet^{-1}=q)=c\bigr)\subset
\sH_\bullet\bigl(\bchi(D_\bullet^{-1}=q\uprho_{\alpha,i})B\bchi(D_\bullet^{-1}=q\uprho_{\alpha,i})=c\bigr)
\end{equation}
\end{lemma}

\begin{proof}
Fix $t\in\RR$. Since $U^\alpha_{i,i}$ is invariant for the scaling group (\cite[Section 1.7]{NTbook}), from Lemma \ref{prev} we know that $B_t\pi_\bullet(U^\alpha_{i,i})B_{-t}=\pi_\bullet\bigl(\tau_{-t}(U^\alpha_{i,i})\bigr)=\pi_\bullet(U^\alpha_{i,i})$. Therefore if
\[
\xi\in\sH_\bullet\bigl(\bchi(D_\bullet^{-1}=q)B\bchi(D_\bullet^{-1}=q)=c\bigr)
\]
then
\[
B\pi_\bullet(U^\alpha_{i,i})\xi=B\pi_\bullet(U^\alpha_{i,i})B^{-1}B\xi=c\pi_\bullet(U^\alpha_{i,i})\xi
\]
(cf.~Corollary \ref{DBD}) and \eqref{zaw} follows (note that there are no domain issues because we are restricting to finite-dimensional eigenspaces of $D_\bullet^{-1}$ for the eigenvalues $q$ and $q\uprho_{\alpha,i}$).
\end{proof}

\begin{theorem}\label{thm}
The set
\begin{equation}\label{Deltaq}
\bigcup\limits_{q\in\spec(D_\bullet^{-1})}\Delta_q
\end{equation}
is finite.
\end{theorem}

\begin{proof}
First let us choose $\alpha\in\Irr(\GG)$ and $i\in\{1,\dotsc,n_\alpha\}$ such that $\uprho_{\alpha,i}>1$ (this is possible because $\GG$ is not of Kac type and $\Tr(\rho_\alpha)=\Tr(\rho_\alpha^{-1})$).

We have the decomposition of $\sH_\bullet$ into eigenspaces of the positive compact operator $D_\bullet^{-1}$:
\[
\sH_\bullet=\bigoplus_{q\in\spec(D_\bullet^{-1})}\sH_\bullet(D_\bullet^{-1}=q).
\]
Consider $\eta\in\ker\pi_\bullet(U^\alpha_{i,i})$ with decomposition
\[
\eta=\sum\limits_{q\in\spec(D_\bullet^{-1})}\eta_q
\]
with $\eta_q\in\sH_\bullet(D_\bullet^{-1}=q)$. Now
\[
0=\pi_\bullet(U^\alpha_{i,i})\eta=\sum_{q\in\spec(D_\bullet^{-1})}\pi_\bullet(U^\alpha_{i,i})\eta_q
\]
and by Lemma \ref{lemShift} each summand is orthogonal to the remaining ones. It follows that $\pi_\bullet(U^\alpha_{i,i})\eta_q=0$ for all $q$ and consequently
\begin{equation}\label{decomkerpiU}
\ker\pi_\bullet(U^\alpha_{i,i})=
\bigoplus_{q\in\spec(D_\bullet^{-1})}\ker\pi_\bullet(U^\alpha_{i,i})\cap\sH_\bullet(D_\bullet^{-1}=q).
\end{equation}
But $\pi_\bullet(U^\alpha_{i,i})$ is a Fredholm operator, so its kernel is finite-dimensional. In particular (since the summands on the right hand side of \eqref{decomkerpiU} are pairwise orthogonal) there exists $q_0$ in $\spec(D_\bullet^{-1})$ such that $\pi_\bullet(U^\alpha_{i,i})$ is injective on $\sH_\bullet(D_\bullet^{-1}=q)$ for all $q\in\spec(D_\bullet^{-1})$ such that $q<q_0$.

Clearly, since there are only finitely many eigenvalues of $D_\bullet^{-1}$ grater than $q_0$ and each is of finite multiplicity, the set
\begin{equation}\label{Deltaq0}
\bigcup_{\substack{q\in\spec(D_\bullet^{-1})\\q\geq{q_0}}}\Delta_q
\end{equation}
is finite. Therefore, if \eqref{Deltaq} is infinite, there exists $\tilde{c}>0$ such that
\[
\tilde{c}\in\Delta_{\tilde{q}}=\spec\bigl(\bchi(D_\bullet^{-1}=\tilde{q})B\bchi(D_\bullet^{-1}=\tilde{q})\bigr)
\]
for some $\tilde{q}<q_0$ and $\tilde{c}$ does not belong to \eqref{Deltaq0}.

Consider now a unit vector $\xi\in\sH_\bullet\bigl(\bchi(D_\bullet^{-1}=\tilde{q})B\bchi(D_\bullet^{-1}=\tilde{q})\bigr)$. For $k\in\ZZ_+$ we have
\[
\pi_\bullet(U^\alpha_{i,i})^k\xi\in\sH_\bullet\bigl(\bchi(D_\bullet^{-1}=\tilde{q}\uprho_{\alpha,i}^{k})B\bchi(D_\bullet^{-1}=\tilde{q}\uprho_{\alpha,i}^{k})=\tilde{c}\bigr).
\]
As $k$ increases $\tilde{q}\uprho_{\alpha,i}^{k}$ tends to infinity, so we let $\tilde{k}=\max\bigl\{k\in\ZZ_+\,\bigr|\bigl.\,\tilde{q}\uprho_{\alpha,i}^{k}<q_0\bigr\}$.

We have
\begin{align*}
\pi_\bullet(U^\alpha_{i,i})\bigl(\pi_\bullet(U^\alpha_{i,i})^{\tilde{k}}\xi\bigr)
&=\pi_\bullet(U^\alpha_{i,i})^{\tilde{k}+1}\xi\\
&\in\sH_\bullet\bigl(\bchi(D_\bullet^{-1}=\tilde{q}\uprho_{\alpha,i}^{(\tilde{k}+1)})B\bchi(D_\bullet^{-1}=\tilde{q}\uprho_{\alpha,i}^{(\tilde{k}+1)})=\tilde{c}\bigr).
\end{align*}
But the latter subspace is $\{0\}$ because $\tilde{q}\uprho_{\alpha,i}^{\tilde{k}}\geq{q_0}$ and $\tilde{c}$ is not in the spectrum of $\bchi(D_\bullet^{-1}=q)B\bchi(D_\bullet^{-1}=q)$ for $q\geq{q_0}$. This contradicts injectivity of $\pi_\bullet(U^\alpha_{i,i})$ on $\sH_\bullet(D_\bullet^{-1}=\tilde{q}\uprho_{\alpha,i}^{\tilde{k}})$, since
\[
\pi_\bullet(U^\alpha_{i,i})^{\tilde{k}}\xi\in\sH_\bullet\bigl(\bchi(D_\bullet^{-1}=\tilde{q}\uprho_{\alpha,i}^{\tilde{k}})B\bchi(D_\bullet^{-1}=\tilde{q}\uprho_{\alpha,i}^{\tilde{k}})=\tilde{c}\bigr)
\subset\sH_\bullet(D_\bullet^{-1}=\tilde{q}\uprho_{\alpha,i}^{\tilde{k}})
\]
because $B$ preserves the decomposition of $\sH_\bullet$ into eigenspaces of $D_\bullet^{-1}$ (Corollary \ref{DBD}).
\end{proof}

\begin{corollary}\label{final}
The quantum group $\GG$ is of Kac type.
\end{corollary}

\begin{proof}
Recall from Corollary \ref{DBD} that
\begin{align*}
B&=\bigoplus_{q\in\spec(D_\bullet^{-1})}\bchi(D_\bullet^{-1}=q)B\bchi(D_\bullet^{-1}=q)\\
&=\bigoplus_{q\in\spec(D_\bullet^{-1})}
\bigoplus_{c\in\Delta_q}c\,\bchi\bigl(\bchi(D_\bullet^{-1}=q)B\bchi(D_\bullet^{-1}=q)=c\bigr).
\end{align*}
Therefore Theorem \ref{thm} implies that $B$ and $B^{-1}$ are bounded with
\[
\|B\|=\sup_{q\in\spec(D_\bullet^{-1})}\sup_{c\in\Delta_q}c<+\infty\quad\text{and}\quad
\|B^{-1}\|=\sup_{q\in\spec(D_\bullet^{-1})}\sup_{c\in\Delta_q}c^{-1}<+\infty.
\]
By Lemma \ref{prev} for any $a\in\C(\GG)$
\[
\pi_\bullet\bigl(\tau_t(a)\bigr)=B_{-t}\pi_\bullet(a)B_t=B^{-\ii{t}}\pi_\bullet(a)B^{\ii{t}},\qquad{t}\in\RR,
\]
so for $a\in\Pol(\GG)$ the holomorphic continuation of the function $t\mapsto\pi_\bullet\bigl(\tau_t(a)\bigr)$ (cf.~\cite[Theorem 2.6]{cqg}, \cite[Page 32]{NTbook}) to $t=-\ii$ is given by $\pi_\bullet\bigl(\tau_{-\ii}(a)\bigr)=B^{-1}\pi_\bullet(a)B$ and hence
\[
\bigl\|\pi_\bullet\bigl(\tau_{-\ii}(a)\bigr)\bigr\|\leq\|B^{-1}\|\|a\|\|B\|.
\]

Thus for any $\alpha\in\Irr(\GG)$ and $i,j\in\{1,\dotsc,n_\alpha\}$
\[
\uprho_{\alpha,i}\uprho_{\alpha,j}^{-1}\|U^\alpha_{i,j}\|
=\|\uprho_{\alpha,i}\uprho_{\alpha,j}^{-1}U^\alpha_{i,j}\|
=\bigl\|\tau_{-\ii}(U^\alpha_{i,j})\bigr\|\leq\|B^{-1}\|\|U^\alpha_{i,j}\|\|B\|,
\]
so that
\[
\uprho_{\alpha,i}\uprho_{\alpha,j}^{-1}\leq\|B^{-1}\|\|B\|
\]
which implies that $\GG$ is a compact quantum of Kac type (cf.~\cite[Remarks after Example 1.7.10]{NTbook}).
\end{proof}

As we already mentioned in the introduction, the assumption that there is a compact quantum group $\GG$ such that the \cst-algebra $\C(\GG)$ is isomorphic to $\C(\UU)\cong\cT$ leads to the contradiction between the relatively easy conclusion that $\GG$ cannot be of Kac type (Section \ref{first}) and the conclusion of Corollary \ref{final} that $\GG$ is of Kac type. It follows that no such compact quantum group exists.

\subsection*{Acknowledgements}\hspace*{\fill}

Research presented in this paper was partially supported by the Polish National Agency for the Academic Exchange, Polonium grant PPN/BIL/2018/1/00197 as well as by the FWO–PAS project VS02619N: von Neumann algebras arising from quantum symmetries. The first author was partially supported by the NCN (National Centre of Science) grant 2014/14/E/ST1/00525.


\end{document}